\documentclass[11pt]{article}

\usepackage{amsmath, amssymb, amsthm, amstext, graphicx,tikz}
\usepackage{enumerate}
\usepackage{fullpage}

\newtheorem{theorem}{Theorem}[section]
\newtheorem{lemma}[theorem]{Lemma}

\newtheorem{problem}[theorem]{Problem}

\newtheorem{corollary}[theorem]{Corollary}
\newtheorem{proposition}[theorem]{Proposition}

\newcommand{\rord}{r_{<}}

\begin{document}

\title{Ordered Ramsey numbers}
\author{David Conlon\thanks{Mathematical Institute, Oxford OX2 6GG,
United Kingdom. Email: {\tt david.conlon@maths.ox.ac.uk}. Research
supported by a Royal Society University Research Fellowship.}\and
Jacob Fox\thanks{Department of Mathematics, Stanford University, Stanford, CA 94305. Email: {\tt jacobfox@stanford.edu}. Research supported by a Packard Fellowship, by NSF Career Award DMS-1352121 and by an Alfred P. Sloan Fellowship.}
\and Choongbum Lee\thanks{Department of Mathematics, MIT, Cambridge, MA
02139-4307. Email: {\tt cb\_lee@math.mit.edu}. Research supported by NSF Grant DMS-1362326}
\and
Benny Sudakov\thanks{Department of Mathematics, ETH, 8092 Zurich, Switzerland.
Email: {\tt benjamin.sudakov@math.ethz.ch}. Research supported
in part by SNSF grant 200021-149111.}}

\date{}

\maketitle

\begin{abstract}
Given a labeled graph $H$ with vertex set $\{1, 2,\ldots,n\}$, the ordered Ramsey number $\rord(H)$ is the minimum $N$ such that every two-coloring of the edges of the complete graph on $\{1, 2, \ldots,N\}$ contains a copy of $H$ with vertices appearing in the same order as in $H$. The ordered Ramsey number of a labeled graph $H$ is at least the Ramsey number $r(H)$ and the two coincide for complete graphs. However, we prove that even for matchings there are labelings where the ordered Ramsey number is superpolynomial in the number of vertices. Among other results, we also prove a general upper bound on ordered Ramsey numbers which implies that there exists a constant $c$ such that $\rord(H) \leq r(H)^{c \log^2 n}$ for any labeled graph $H$ on vertex set $\{1,2, \dots, n\}$. 
\end{abstract}

\section{Introduction} 

Given a graph $H$, the {\it Ramsey number} $r(H)$ is defined to be the smallest natural number $N$ such
that every two-coloring of the edges of $K_N$ contains a monochromatic copy of $H$. That these
numbers exist was first proved by Ramsey~\cite{R30} and rediscovered independently by Erd\H{o}s and Szekeres~\cite{ES35}. 

The most famous question in graph Ramsey theory is that of estimating the Ramsey number $r(K_n)$ of the complete graph $K_n$ on $n$ vertices. However, despite some smaller order improvements~\cite{C09, S75}, the standard estimates~\cite{E47, ES35} that $2^{n/2} \leq r(K_n) \leq 2^{2n}$ have remained largely unchanged for nearly seventy years. After the complete graph, the next most classical topic in the area is the study of Ramsey numbers of sparse graphs, that is, graphs with certain upper bound constraints on the degrees of their vertices. This direction was pioneered by Burr and Erd\H{o}s~\cite{BE75} in 1975 and the topic has since played a central role in graph Ramsey theory. 

Answering a question of Burr and Erd\H{o}s, Chv\'atal, R\"odl, Szemer\'edi and Trotter \cite{CRST83} proved that for every $\Delta$ there is $c(\Delta)$ such that every graph $H$ on at most $n$ vertices with maximum degree $\Delta$ satisfies $r(H) \leq c(\Delta)n$. That is, Ramsey numbers of bounded-degree graphs grow linearly in the number of vertices. Another stronger conjecture of Burr and Erd\H{o}s remained open until very recently. We say that a graph is $d$-degenerate if every subgraph has a vertex of degree at most $d$. Equivalently, a graph is $d$-degenerate if there is an ordering $v_1, v_2, \ldots,v_n$ of its vertices such that each vertex $v_i$ has at most $d$ neighbors $v_j$ with $j<i$. Burr and Erd\H{o}s conjectured that for every $d$ there is $c(d)$ such that every $d$-degenerate graph $H$ on at most $n$ vertices satisfies $r(H) \leq c(d)n$. Building on earlier work by several authors~\cite{AKS03, FS09, FS092, KR04, KS03}, this conjecture has now been solved by Lee~\cite{L15}.

 
In this paper, we will study analogues of these results for ordered graphs. An {\it ordered graph} or {\it labeled graph} $H$ on $n$ vertices is a graph whose vertices have been labeled with $\{1, 2, \dots, n\}$. An ordered graph $G$ on $[N] : = \{1, 2, \dots, N\}$ {\it contains} an ordered graph $H$ on $[n]$ if there is a mapping $\phi:[n] \rightarrow [N]$ such that $\phi(i)<\phi(j)$ for $1 \leq i < j \leq n$ and $(\phi(i),\phi(j))$ is an edge of $G$ whenever $(i,j)$ is an edge of $H$. Given an ordered graph $H$, the {\it ordered Ramsey number} $\rord (H)$ is defined to be the smallest natural number $N$ such that every two-coloring of the edges of the complete graph on $[N]$ contains a monochromatic ordered copy of $H$. The natural analogue of this function for $q$ colors will be denoted by $\rord(H ; q)$.

In some sense, the study of ordered Ramsey numbers is as old as Ramsey theory itself. One of the most famous results in the classic 1935 paper of Erd\H{o}s and Szekeres~\cite{ES35} states that every sequence $x_1, x_2, \ldots,x_r$ of $r \geq (n-1)^2+1$ distinct real numbers contains an increasing  or decreasing subsequence of length $n$. To prove this, consider a red/blue-coloring of the edges of the complete graph on vertex set $[r]$ where $(i,j)$ with $i<j$ is red if $x_i<x_j$ and blue otherwise. Note that a subsequence is increasing or decreasing if and only if it forms a monotone monochromatic path in this edge coloring. We may therefore assume that there is no monotone red path of length $n$. We now label each vertex by the length of the longest monotone red path ending at that vertex. It is easy to see that any set of vertices with the same label forms a blue clique. Therefore, since there are only $n-1$ possible labels and $r \geq (n-1)^2 + 1$ vertices, we may conclude that there exists a blue clique of order at least $n$ and the result of Erd\H{o}s and Szekeres follows. In proving this result, we have also shown that $r_<(P_n) \leq (n-1)^2 + 1$, where $P_n$ is the monotone path with $n$ vertices. Since this is easily seen to be tight, we have $r_<(P_n) = (n-1)^2 + 1$. More generally, for $q$ colors instead of $2$, the $q$-color ordered Ramsey number of the monotone path satisfies $\rord(P_n;q)=(n-1)^{q}+1$.

Another foundational result in Ramsey theory, known as the happy ending theorem, also has a natural proof using ordered Ramsey numbers. This result, again due to Erd\H{o}s and Szekeres~\cite{ES35}, states that for each positive integer $n$ there is an integer $N$ such that every set of $N$ points in the plane in general position (that is, with no three on a line) contains $n$ points which are the vertices of a convex polygon. We write $g(n)$ for the smallest such $N$. To estimate $g(n)$, it will be useful to generalize the concept of ordered Ramsey numbers to hypergraphs. That is, given an ordered $k$-uniform hypergraph $\mathcal{H}$ and a positive integer $q$, we let $\rord(\mathcal{H};q)$ be the smallest natural number $N$ such that every $q$-coloring of the edges of the complete $k$-uniform hypergraph on $[N]$ contains a monochromatic ordered copy of $\mathcal{H}$.

Suppose now that we have $N$ points in the plane in general position. By rotating the plane if necessary, we may assume that no two points are on a vertical line. Denote the $N$ points by  $p_i=(x_i,y_i)$ with $x_1<\ldots<x_N$. A set $P$ of points in the plane forms a {\it cup} (respectively, {\it cap}) if the points in $P$ lie on the graph of a convex (respectively, concave) function. Note that both cups and caps form convex polygons. Consider a red/blue-coloring of the edges of the complete $3$-uniform hypergraph on $\{1,2,\ldots,N\}$ where $\{i,j,k\}$ with $i<j<k$ is red if $\{p_i,p_j,p_k\}$ form a cup and blue if $\{p_i,p_j,p_k\}$ form a cap. If we write $P^{(k)}_n$ for the monotone $k$-uniform tight path on $\{1,2,\ldots,n\}$,  where $\{i,i+1,\ldots,i+k-1\}$ is an edge for $1 \leq i \leq n-k+1$, then we see that a sequence of points forms a cap or a cup if and only if the corresponding vertices form a monochromatic ordered copy of $P^{(3)}_n$. Hence, $g(n) \leq \rord(P^{(3)}_n)$, which is known to equal ${2n-4 \choose n-2}+1$. 

An extension of the happy ending theorem proved by Pach and T\'oth~\cite{PT00} shows that for every positive integer $n$ there is an integer $N$ such that every set of $N$ convex sets in the plane in general position contains $n$ convex sets which are in convex position. We write $h(n)$ for the smallest such $N$. It was shown by Fox, Pach, Sudakov and Suk~\cite{FPSS12} that $h(n) \leq \rord(P^{(3)}_n; 3)$, leading the authors to study the growth of ordered hypergraph Ramsey numbers for monotone paths. The results of~\cite{FPSS12}, and subsequent improvements made by Moshkovitz and Shapira~\cite{MS13}, show that for $k \geq 3$ the ordered Ramsey number $\rord(P^{(k)}_n; q)$ grows as a $(k-2)$-fold exponential in $n^{q-1}$. This is in stark contrast to the classical unordered problem, where the Ramsey number of $P_n^{(k)}$ grows linearly in the number of vertices for all uniformities $k$. More recently, ordered Ramsey numbers for tight paths were also used by Milans, Stolee and West~\cite{MSW} to give bounds on the minimum number of interval graphs whose union is the line graph of $K_n$. 

While there has been much progress on understanding the ordered Ramsey numbers of monotone paths, there has been surprisingly little work on more general ordered graphs and hypergraphs. In this paper, we attempt to bridge this gap by conducting a more systematic study of ordered Ramsey numbers, focusing on the case of graphs. We note that a similar study was conducted independently by Balko, Cibulka, Kr\'al and Kyn\v cl \cite{BCKK14} and that there are overlaps between many of our results.

One of the more striking aspects of the discussion above is the vast difference between the usual Ramsey number and the ordered Ramsey number for monotone paths of high uniformity. Our first result shows that a large gap exists already for graphs, even when the graph is as simple as a matching, where the ordinary Ramsey number is clearly linear. Here and throughout the paper, all logs are taken to base $2$.

\begin{theorem} \label{thm:intromatchlower}
There exists a positive constant $c$ such that, for all even $n$, there exists an ordered matching $M$ on $n$ vertices with
\[\rord(M) \geq n^{c \log n/\log \log n}.\]
\end{theorem}

This lower bound actually holds for almost every ordering of a matching on $n$ vertices, so that Theorem~\ref{thm:intromatchlower} represents typical rather than atypical behavior. An almost matching upper bound is provided by the following simple theorem (which, in a slightly weaker form, also follows from a result of Cibulka, Gao, Kr\v c\'al, Valla and Valtr \cite{CGKVV13}).

\begin{theorem} \label{thm:intromatchupper}
For any ordered matching $M$ on $n$ vertices,
\[\rord(M) \leq n^{\lceil \log n \rceil}.\]
\end{theorem}

More generally, we can prove that there exists a constant $c$ such that for any ordered graph $H$ on $n$ vertices with degeneracy $d$, $\rord(H) \leq n^{c d \log n}$, where the degeneracy of an ordered graph $H$ is the smallest $d$ for which the corresponding unordered graph is $d$-degenerate. This is a special case of an even more general result. To state this result, we define the {\it interval chromatic number} $\chi_< (H)$ of an ordered graph $H$ to be the minimum number of intervals into which the vertex set of $H$ may be partitioned so that no two vertices in the same interval are adjacent. This is similar to the notion of chromatic number but now the independent sets must also be intervals in the given ordering. For any graph $H$, it is easy to see that there is an ordering of the vertices of $H$ such that the interval chromatic number is the same as the chromatic number.

Interval chromatic number plays a key role in the study of extremal problems on ordered graphs (see, for example, \cite{PT06}). In particular, Pach and Tardos \cite{PT06} observed that the maximum number $\textrm{ex}_{<}(n,H)$ of edges an ordered graph on $n$ vertices can have without containing the ordered graph $H$ is  
\[\left(1 - \frac{1}{\chi_<(H) - 1} + o(1)\right) \binom{n}{2}.\] 
The following result shows that it also plays a fundamental role in ordered Ramsey theory.

\begin{theorem} \label{thm:introBEwithChi}
There exists a constant $c$ such that for any ordered graph $H$ on $n$ vertices with degeneracy $d$ and interval chromatic number $\chi$, 
\[\rord(H) \leq n^{c d \log \chi}.\]
\end{theorem}

In particular, this result implies that there are orderings under which the ordered Ramsey number of a $d$-degenerate graph with $n$ vertices is polynomial in $n$.
The following result shows that restricting the interval chromatic number is still not enough to force the ordered Ramsey number to be linear, even for matchings. It is also close to best possible, since an elementary argument (see Section~\ref{sec:matchings}) shows that $r_<(M) \leq n^2$ for any matching $M$ with $n$ vertices and interval chromatic number $2$.

\begin{theorem} \label{thm:introlowerwithChi}
There exists a positive constant $c$ such that, for all even $n$, there exists an ordered matching $M$ on $n$ vertices with interval chromatic number $2$ and
\[\rord(M) \geq \frac{c n^2}{\log^2 n \log \log n}.\]
\end{theorem}

So far, our results have focused on very sparse graphs. For denser graphs, the ordered Ramsey number behaves more like the usual Ramsey number. Indeed, we have the following result, which generalizes a result of the first author \cite{C13} in two ways: firstly, by estimating the ordered Ramsey number rather than just the usual Ramsey number and, secondly, by replacing maximum degree with degeneracy.

\begin{theorem} \label{thm:introdense}
There exists a constant $c$ such that for any ordered graph $H$ on $n$ vertices with degeneracy at most $d$,
\[\rord(H) \leq 2^{c d \log^2(2n/d)}.\]
\end{theorem}

This result is close to sharp when $d$ is very large and also, by Theorem~\ref{thm:intromatchlower}, when $d$ is very small. We note two simple corollaries of this theorem. The first, proved for the usual Ramsey number in~\cite{C13}, says that for any graph on $n$ vertices with $o(n^2)$ edges, the ordered Ramsey number is $2^{o(n)}$. This follows by noting that any graph on $n$ vertices with at most $\delta n^2$ edges has degeneracy at most $(2 \delta)^{1/2} n$ and substituting into Theorem~\ref{thm:introdense}. More precisely, we have the following result.

\begin{corollary}
There exists a constant $c$ such that for any ordered graph $H$ on $n$ vertices with at most $\delta n^2$ edges,
\[r_<(H) \leq 2^{c \sqrt{\delta} \log^2(1/\delta) n}.\]
\end{corollary}

Since any graph $H$ with degeneracy at least $d$ contains a subgraph of minimum degree at least $d$, a simple application of the probabilistic method implies that the Ramsey number of $H$ must be at least $2^\frac{d}{2}$. Using also the elementary observation that $r(H) \geq n$, we have the following corollary.

\begin{corollary}
There exists a constant $c$ such that for any ordered graph $H$ on $n$ vertices with degeneracy $d$,
\[r_<(H) \leq r(H)^{c \gamma(H)},\]
where $\gamma(H) = \min\{\log^2(n/d), d \log(n/d)\}$.
\end{corollary}

It is also interesting to define an off-diagonal variant of the ordered Ramsey number. Given two ordered graphs $G$ and $H$, we define the ordered Ramsey number $r_<(G, H)$ to be the smallest natural number $N$ such that any two-coloring of the edges of the complete graph on vertex set $[N]$ in red and blue, say, contains either a red ordered copy of $G$ or a blue ordered copy of $H$. To give an example, we note that the proof of the Erd\H{o}s--Szekeres theorem on monotone subsequences given earlier actually shows that
\begin{align} \label{eq:es_offdiagonal}
	\rord(P_m, K_n) = (m-1)(n-1) + 1.
\end{align}

In the classical case, the most intensively studied off-diagonal Ramsey number is $r(K_n, K_3)$. It is easy to see that $r(K_n, K_3) \leq n^2$ and a well-known result of Ajtai, Koml\'os and Szemer\'edi \cite{AKS80} improves this to $r(K_n, K_3) = O(\frac{n^2}{\log n})$. In a remarkable breakthrough, Kim \cite{K95} showed that this upper bound is essentially tight, that is, that $r(K_n, K_3) = \Omega(\frac{n^2}{\log n})$. Recent advances in the study of triangle-free processes \cite{BK14, FGM14} have led to further improvements in these bounds, so that $r(K_n, K_3)$ is now known up to an asymptotic factor of $4$. 

The main off-diagonal problem that we treat in the ordered case is that of determining the ordered Ramsey number $r_<(M, K_3)$, where $M$ is an ordered matching. For the upper bound, we could only prove the trivial estimate $r_<(M, K_3) \leq r_<(K_n, K_3) = O(\frac{n^2}{\log n})$. For the lower bound, we have the following result, which improves considerably on the trivial linear bound.

\begin{theorem} \label{thm:offdiag}
There exists a positive constant $c$ such that, for all even $n$, there exists an ordered matching $M$ on $n$ vertices with
\[r_{<}(M, K_3) \ge c\Big(\frac{n}{\log n}\Big)^{4/3}.\]
\end{theorem}

In the next section, we will study the ordered Ramsey number of matchings, proving Theorems~\ref{thm:intromatchlower}, \ref{thm:intromatchupper} and \ref{thm:introlowerwithChi}. We will also show that the ordered Ramsey number of matchings with bounded bandwidth is at most polynomial in the number of vertices. In Section~\ref{sec:gen}, we will prove Theorems~\ref{thm:introBEwithChi} and \ref{thm:introdense}. The proof of Theorem~\ref{thm:offdiag} is given in Section~\ref{sec:offdiag}. We conclude, in Sections \ref{sec:conclusion1} and \ref{sec:conclusion2}, with some further remarks and a collection of open problems. In particular, we observe a connection between ordered Ramsey numbers and hypergraph Ramsey numbers. We also classify those graphs which have linear ordered Ramsey number in every ordering. Throughout the paper, we omit floor and ceiling signs whenever they are not essential. We also do not make any serious attempt to optimize absolute constants in our statements and proofs.

\section{Matchings} \label{sec:matchings}

We begin with a simple upper bound for the ordered Ramsey number of matchings in terms of the number of vertices $n$ and the interval chromatic number $\chi$. By saying that the $\chi$-partite graph $K_{n', n', \dots, n'}$ is trivially ordered, we mean that the vertices of each partite set appear as an interval.

\begin{theorem} \label{thm:matchupper}
Let $M$ be an ordered matching on $[n]$ and let $K = K_{n', n', \dots, n'}$ be a trivially ordered $\chi$-partite graph with $\chi \geq 2$. Then 
\[r_<(M, K) \leq n^{\lceil \log \chi \rceil} n'.\]
In particular, if $M$ is an ordered matching on $n$ vertices of interval chromatic number $\chi$, then
\[r_<(M) \leq n^{\lceil \log \chi \rceil+1}.\]
\end{theorem}

\begin{proof}
It suffices to prove the statement when $\chi = 2^j$ for positive integers $j$. We prove this by induction on $j$. For the base case $j = 1$, we have to show that $r_<(M, K_{n',n'}) \leq n n'$. To see this, suppose that we are given a two-coloring of the edges of $K_{N_1}$ with $N_1 = nn'$.  Partition $[N_1]$ into $n$ consecutive intervals $V_{1}, V_2, \dots, V_{n}$, each of length $n'$. We try to embed a red copy of $M$ by placing vertex $i$ in the set $V_i$ for each $i = 1, 2, \dots, n$. If this procedure does not produce a red copy of $M$, then there are two indices $i_1$ and $i_2$ such that every edge between $V_{i_1}$ and $V_{i_2}$ is blue. That is, we find either a red copy of $M$ or a blue copy of $K_{n',n'}$.

Let $N_j = n^j n'$ and suppose that any two-coloring of the edges of $[N_{j-1}]$ contains either a red copy of $M$ or a blue copy of the trivially ordered $\chi$-partite graph $K_{n', n', \dots, n'}$ with $\chi = 2^{j-1}$. Partition $[N_{j}]$ into $n$ consecutive intervals $V_{1}, V_2, \dots, V_{n}$, each of length $N_{j-1}$.  We again try to embed a red copy of $M$ by placing vertex $i$ in the set $V_i$ for each $i = 1, 2, \dots, n$. If this procedure does not produce a red copy of $M$, then there are two indices
$i_1$ and $i_2$ such that every edge between $V_{i_1}$ and $V_{i_2}$ is blue. Moreover, by the induction hypothesis, either one of $V_{i_1}$ or $V_{i_2}$ contains a red copy of $M$, in which case we are done, or both $V_{i_1}$ and $V_{i_2}$ contain a blue copy of the $2^{j-1}$-partite graph $K_{n',n', \dots, n'}$. Combining the two gives a blue copy of the $2^j$-partite graph $K_{n',n',\dots,n'}$.
\end{proof}

Since any ordered matching $M$ on $n$ vertices is a subgraph of $K_n = K_{1,1, \dots, 1}$, the $n$-partite graph with parts of size $1$, we see that $r_<(M) \leq r_<(M, K_n) \leq n^{\lceil \log n \rceil}$, establishing Theorem~\ref{thm:intromatchupper}. We will now prove Theorem~\ref{thm:intromatchlower} by showing that $\rord(M) \geq n^{c \log n/\log \log n}$ for almost all orderings. We need a simple lemma which says that in a randomly ordered matching with $n$ vertices, that is, an ordered matching chosen uniformly at random from the set of all possible orderings of a matching with $n$ vertices, any two disjoint intervals of length at least $4\sqrt{n \log n}$ have an edge between them.

\begin{lemma} \label{lem:jumbprop}
Let $M$ be a random matching on vertex set $[n]$. Then, asymptotically almost surely, there is an edge between any two disjoint intervals of length at least $4 \sqrt{n \log n}$.
\end{lemma}

\begin{proof}
Given two disjoint sets $A$ and $B$ of order $t$, where $t$ is even, the probability that a random matching has no edges between $A$ and $B$ is at most
\[
	\left(\frac{n-t-1}{n-1}\right) \left( \frac{n-t-3}{n-3} \right) \cdots \left( \frac{n-2t+1}{n-t+1} \right) \leq \left(\frac{n-t}{n}\right)^{t/2} \leq e^{-t^2/2n},
\]
where the first inequality holds since $\frac{n-t-k}{n-k} < \frac{n-t}{n}$ for all $k > 0$ and the second inequality follows from the fact that $1 - x \le e^{-x}$ for all $x$. Since there are at most $n^2$ pairs of intervals of length $t$, the probability that there exist two intervals of length $t$ with no edge between them is at most $n^2 e^{-t^2/2n}$. For $t \geq 3 \sqrt{n\log n}$, this tends to zero with $n$. Therefore, taking $t$ to be an even integer between $3 \sqrt{n \log n}$ and $4 \sqrt{n \log n}$, we see that asymptotically almost surely any two disjoint intervals of length at least $t$ have an edge between them, completing the proof.
\end{proof}

\begin{theorem} \label{thm:matchingrandom}
Let $M$ be a random matching on vertex set $[n]$. Then, asymptotically almost surely, 
\[\rord(M) \ge n^{\log n/20\log\log n}.\]
\end{theorem}

\begin{proof} 
By Lemma~\ref{lem:jumbprop}, we may assume that in $M$ any two intervals of length at least $4 \sqrt{n \log n}$ have an edge between them.
Let $s = \lfloor n^{1/4} \rfloor$ and suppose that $c$ is a two-coloring of the edges of the complete graph
on vertex set $[s]$ without a monochromatic clique of order
$\log n$ (such a coloring exists by Ramsey's theorem). 

Let $G_{0}$ be the graph with a single vertex. For $i\ge1$, we recursively
define edge-colored ordered complete graphs $G_{i}$ as follows. Let $G_{i,1},G_{i,2},\dots,G_{i, s}$
be vertex disjoint copies of $G_{i-1}$. We form $G_i$ by placing these copies
of $G_{i,j}$ in sequential order, that is, placing the vertices of $G_{i,j}$ before the vertices of
$G_{i,j+1}$, and, for each $1 \leq j < j' \leq s$, adding a complete bipartite graph in color $c(j,j')$ between the two graphs $G_{i,j}$
and $G_{i,j'}$.

We prove by induction on $i$ that the graph $G_{i}$ does not contain
a monochromatic copy of a subgraph of $M$ induced on a subinterval of size at least $n^{3/4}\cdot(8\log n)^{i}$ (we refer to such subgraphs as {\em subintervals} of $M$).
The claim is trivially true for $i=0$ since $G_{i}$ consists of
a single vertex. Suppose now that the claim has been
proved for $G_{i-1}$ and, for the sake of contradiction, suppose that
$G_{i}$ contains a monochromatic subinterval $M'$ of $M$ of size at least $n^{3/4}\cdot(8\log n)^{i}$.
Abusing notation, we also let $M'$ denote this copy in $G_{i}$. Without
loss of generality, we may assume that it is a red copy. For $j=1,2,\dots, s$,
let $W_{j}=V(G_{i,j})\cap V(M')$ and note that each $W_j$ forms a subinterval of $M$. We call a set $W_{j}$ \emph{small
}if $|W_{j}| < 4\sqrt{n \log n}$ and \emph{large} otherwise. 

There are at most $4 \sqrt{n \log n} \cdot s \leq 4 n^{3/4} \sqrt{\log n}$ vertices of $M'$ in small
sets $W_{j}$. Therefore, at least $\frac{1}{2}n^{3/4}\cdot(8\log n)^{i}$
vertices lie in large sets $W_{j}$. By our assumption, there exists a red edge of $M$ between every pair of large
sets. However, since $c$ does not contain a monochromatic clique
of order $\log n$, there are fewer than $\log n$
large sets $W_{j}$. Therefore, for some index $\overline{j}$, we have
\[
	|V(G_{i,{\overline{j}}})\cap V(M')|
	\ge \frac{1}{\log n} \cdot \frac{1}{2} n^{3/4} (8\log n)^{i}
	\ge n^{3/4}\cdot(8\log n)^{i-1},
\]
contradicting the induction hypothesis. Since 
\[
n^{3/4}\cdot(8\log n)^{\log n/4\log(8\log n)}=n,
\]
we see that for $t = \lfloor \frac{\log n}{4 \log(8 \log n)} \rfloor$ the graph $G_{t}$
does not contain a monochromatic copy of $M$. Since $G_{t}$ has
$s^{t}$ vertices and, for $n$ sufficiently large, $s^t \geq n^{\frac{\log n}{20\log\log n}}$, 
the claimed lower bound follows.
\end{proof}

The only property of $M$ used in the proof of Theorem~\ref{thm:matchingrandom} is that there is at least one edge between any two disjoint intervals of length $4\sqrt{n\log n}$. Since it is straightforward to construct such matchings explicitly (see the start of Section~\ref{sec:offdiag}), we also have explicit examples of ordered matchings with superlinear ordered Ramsey number.

It follows from Theorem~\ref{thm:matchupper} that if $M$ has interval chromatic number $2$ then $r_<(M) \leq n^2$. We now show that this is close to best possible. We will construct our matching using the well-known van der Corput sequence or rather a collection of permutations derived from the van der Corput sequence. To define these permutations, we take the integers $\{0, 1, \dots, 2^h - 1\}$, write each as a binary expansion with $h$ digits and reverse its expansion. We call the resulting permutation $\pi: [2^h] \rightarrow [2^h]$ a van der Corput permutation. For $n = 2^h$, these permutations are known to have the property that for any pair of intervals $I, J \subseteq [n]$,
\begin{align} \label{eq:vandercorput}
\left| |\pi(I) \cap J|- \frac{|I||J|}{n} \right| \leq C\log n,
\end{align} 
where $C$ is an absolute constant (see, e.g., \cite[Chapter 2]{MGD}). We note that this is a considerably smaller discrepancy than one could hope to achieve using a random permutation. We are now ready to prove Theorem~\ref{thm:introlowerwithChi} in the following form.

\begin{theorem} 
There exists a positive constant $c$ such that for all $n = 2^h$ there is a matching of interval chromatic number $2$ with $2n$ vertices satisfying
\[r_<(M) \geq \frac{c n^2}{\log^2 n \log \log n}.\]
\end{theorem}
 
\begin{proof}
Let $\pi$ be the van der Corput permutation. We take $M$ to be the ordered perfect matching with interval chromatic $2$ on $[2n]$ which matches $i \in [n]$ with $n+\pi(i)$.  

Let $s = \frac{n}{8 \log n}$ and $t = \frac{8 c n}{\log n \log \log n}$, where $c$ is a sufficiently small positive constant. Consider a random red/blue-coloring $\chi$ of the edges of the complete graph with loops on $[t]$. Let $I_i=[(i-1)s+1,is]$ for $1 \leq i \leq t$. Let $\psi$ be the red/blue-coloring of the complete graph on $[ts]$, where every edge between $I_i$ and $I_j$ is of color $\chi(i,j)$. 
We will show that with positive probability the edge coloring $\psi$ does not have a monochromatic ordered copy of $M$ and, therefore, $r_{<}(M) > ts= \frac{c n^2}{\log^2 n \log \log n}$. By symmetry between the two colors, it suffices to show that the probability that the red graph contains an ordered copy of $M$ is less than $1/2$. Note that the red graph is a blow-up of the random graph with loops on $[t]$, where each vertex $i$ is replaced by an interval $I_i$ of length $s$.

If the coloring $\psi$ contains a red copy of $M$, then there exists an integer $k$ with $1 \leq k \leq t$ and partitions $[n]=A_1 \cup \dots \cup A_k$, $[n+1,2n]=B_k \cup B_{k+1} \cup \dots \cup B_t$ of the vertex set of $M$ into intervals such that $A_i$ embeds into $I_i$ for all $1 \leq i \leq k$ and $B_j$ embeds into $I_j$ for all $k \leq j \leq t$. Each $A_i$ and $B_j$ has size at most $s$ and if $\pi(A_i) + n$ contains an element of $B_j$, then $\chi(i,j)$ must be red. We note that there are at most
\[\binom{2n + t}{t} \le\left(e((4c)^{-1} \log n \log \log n + 1)\right)^{t}\le e^{9c n/\log n}\]
different partitions of this form. We now fix such a partition.

Let $d=2\log n$. Let $A_{i_0}$ be the $(d+1)^{th}$ largest interval of the form $A_i$ and $B_{j_0}$ the $(d+1)^{th}$ largest interval of the form $B_j$. If $|A_{i_0}||B_{j_0}| > Cn\log n$, then every interval $A_i$ with size at least $|A_{i_0}|$ and every interval $B_j$ with size at least $|B_{j_0}|$ has an edge of $M$ between them. Therefore, if the partition gave rise to a monochromatic red copy of $M$, $\chi$ would contain a red complete bipartite graph with parts of size $d$. The probability of this event is at most
\[
	2^{-d^2}{t \choose d}^2<2^{-d^2} \frac{t^{2d}}{d!^2} < \frac{1}{4}.
\]

Otherwise, we may suppose that $|A_{i_0}||B_{j_0}| \leq Cn \log n$. 
As each $A_i$ and each $B_j$ has size at most $s$, there are at most $2ds$ edges of $M$ coming from some $A_i$ of size larger than $A_{i_0}$ or some $B_j$ of size larger than $|B_{j_0}|$. Therefore, there are at least $n-2ds \geq n/2$ edges of $M$ between those $A_i$ and $B_j$ with $|A_i| \leq |A_{i_0}|$ and $|B_j|\leq |B_{j_0}|$ and hence between those $A_i$ and $B_j$ with $|A_i||B_j| \leq Cn \log n$. By \eqref{eq:vandercorput}, each such pair $A_i,B_j$ has at most $\frac{|A_i||B_j|}{n}+C\log n \leq 2C\log n$ edges between them and thus there are at least $\frac{n/2}{2C\log n}=\frac{n}{4C\log n}$ pairs $A_i,B_j$ for which there is at least one edge of $M$ between $A_i$ and $B_j$. But this implies that the coloring $\chi$ contains a particular red subgraph with at least $\frac{n}{4C\log n}$ edges, which occurs with probability at most $2^{-n/4C \log n}$.
Since the collection of edges is determined by the choice of the $A_i$ and $B_j$ and there are at most $e^{9c n/\log n}$ choices for these sets, we therefore see that the probability the coloring $\psi$ contains a red copy of such a graph is at most
\[
	e^{9c n/\log n} 2^{-n/4C \log n} < \frac{1}{4},
\] 
for $c$ sufficiently small. Hence, we see that the probability $\psi$ contains a red copy of $M$ is at most $1/2$, completing the proof. 
\end{proof}


The ordered Ramsey number of a matching is not always controlled by the interval chromatic number. For example, despite having interval chromatic number which is linear in $n$, the ordered Ramsey number of the matching with edges $(1, 2), (3,4), \dots, (n-1, n)$ is clearly linear in $n$. More generally, if we know that a matching has bounded bandwidth with respect to the given ordering, that is, there exists a $k$ such that $|i-j| \leq k$ for any edge $(i,j)$, then we can show that the ordered Ramsey number is at most a polynomial whose power is dictated by the bandwidth. The proof of this fact relies on the next lemma. 

Given two ordered graphs $G$ and $H$, we define their ordered lexicographic product $G \cdot H$ to be the graph consisting of $t := |H|$ consecutive ordered copies of $G$, which we call $G_1, G_2, \dots, G_t$, with all vertices of $G_i$ joined to all vertices of $G_j$ if and only if $(i, j)$ is an edge of $H$. 

\begin{lemma} \label{lem:lexprod}
For any ordered matching $M$ and any ordered graphs $G$ and $H$, 
\[r_< (M, G \cdot H) \leq r_<(M, G) \cdot r_< (M, H).\]
\end{lemma}

\begin{proof}
Suppose that we are given a two-coloring of the edges of $K_{N}$ with $N = r_<(M, G) \cdot r_<(M, H)$.  Partition $[N]$ into $s := r_<(M, H)$ consecutive
intervals $V_{1}, V_2, \dots, V_s$, each of length $r_<(M, G)$. We consider the reduced graph with $s$ vertices $v_1, v_2, \dots, v_s$, connecting $v_i$ and $v_j$ in red if there are any edges between $V_i$ and $V_j$ in red and in blue otherwise, that is, if the bipartite graph between $V_i$ and $V_j$ is completely blue. If the reduced graph contains a red ordered copy of $M$ then so does the original graph. We may therefore assume that the reduced graph has a blue ordered copy of $H$ with vertices $v_{i_1}, v_{i_2}, \dots, v_{i_t}$, where $t := |H|$. This gives vertex sets $V_{i_1}, V_{i_2}, \dots, V_{i_t}$ such that there are complete blue bipartite graphs between vertex sets corresponding to edges of $H$. Since $V_{i_j}$ has size $r_<(M,G)$, we see that either some $V_{i_j}$ contains an ordered red copy of $M$, in which case we are done, or every $V_{i_j}$ contains an ordered blue copy of $G$. In the latter case, we may use the blue edges between pieces to get an ordered blue copy of $G \cdot H$, completing the proof.
\end{proof}

The result about bandwidth mentioned earlier is now an easy consequence of this lemma. We say that an ordered graph has {\it bandwidth} at most $k$ if $|i-j| \leq k$ for every edge $(i,j)$. We also write $P_n^k$ for the $k$th power of a path, the ordered graph on $[n]$ where $i$ and $j$ are adjacent if and only if $|i - j| \leq k$. Note that an ordered graph on $[n]$ has bandwidth at most $k$ if and only if it is a subgraph of $P_n^k$.

\begin{theorem} \label{thm:bandwidth}
For any ordered matching $M$ on $n$ vertices and any positive integer $k$,
\[r_< (M, P_n^k) \leq n^{\lceil \log k \rceil + 2}.\]
In particular, if $M$ is an ordered matching on $n$ vertices with bandwidth at most $k$, then
\[r_<(M) \leq n^{\lceil \log k \rceil + 2}.\]
\end{theorem}

\begin{proof}
To begin, note that $P_n^k$ is a subgraph of $K_k \cdot P_n$. By Lemma~\ref{lem:lexprod}, it follows that
\[r_<(M, P_n^k) \leq r_<(M, K_k) \cdot r_<(M, P_n).\]
Theorem~\ref{thm:matchupper} implies that $r_<(M, K_k) \leq n^{\lceil \log k \rceil}$ and observation
\eqref{eq:es_offdiagonal} from the introduction implies that $r_<(M,P_n) \le r_<(K_n,P_n) \le n^2$. The result follows.
\end{proof}

\section{General graphs} \label{sec:gen}

In this section, we will prove Theorems~\ref{thm:introBEwithChi} and~\ref{thm:introdense}. We begin with Theorem~\ref{thm:introBEwithChi}, which is contained in the following result.

\begin{theorem} \label{thm:uppergen}
\label{thm:degenerate}Let $H$ be an ordered $d$-degenerate graph on $n$ vertices with maximum degree $\Delta$ and let
$K=K_{n',n',\cdots,n'}$ be a trivially ordered complete $\chi$-partite graph. Let $s=\left\lceil \log \chi\right\rceil$ 
and $D=8\chi^2 n'$. Then 
\[r_{<}(H,K)\le2^{s^2 d + s} \Delta^s n^{s} D^{ds + 1}.\]
In particular, if $H$ is an ordered $d$-degenerate graph on $n$ vertices with interval
chromatic number $\chi$, then 
\[r_{<}(H)\le n^{32 d\log\chi}.\]
\end{theorem}

The next lemma is the key technical step in the proof of Theorem~\ref{thm:uppergen}. It says that if a graph on $[N]$ does not contain a copy of a particular ordered $d$-degenerate graph $H$ then it contains an ordered collection of large subsets with low density between each pair of subsets. In the unordered case, analogues of this result may be found in~\cite{FS08, GRR}. To prove the lemma, we will attempt to embed $H$ greedily. If this fails, we can show that there must be two large vertex sets, say $A$ and $B$, with low density between them. We then repeat this procedure inside the two vertex sets $A$ and $B$. If there is no copy of $H$ in $A$, there will be two large vertex subsets $A_1$ and $A_2$ with low density between them and, similarly, if there is no copy of $H$ in $B$, there will be two large vertex subsets $B_1$ and $B_2$ with low density between them. Therefore, we have found four large vertex subsets with low density between each pair of subsets. Iterating this procedure yields the result below.

\begin{lemma} \label{lem:deg}
Let $H$ be an ordered $d$-degenerate graph on $n$ vertices with maximum degree $\Delta$. Suppose that a real number $0 < c < 1$ and a positive integer $s \geq 1$ are given and that $N \geq (2\Delta n (2^{s} c^{-1})^{d}\big)^{s}$. If an ordered graph on vertex set $[N]$ does not contain an ordered copy of $H$, then there exist sets $W_{1},W_{2},\dots,W_{2^{s}}\subset[N]$ such that 
\begin{itemize}
\item[(i)] 
for all $i$, $|W_i| \ge\frac{c^{sd} N}{(2^{sd + 1} \Delta n)^{s}}$,

\item[(ii)] 
for $i<j$, all vertices in $W_{i}$ precede all vertices in $W_{j}$,

\item[(iii)] 
for $i<j$, the density of edges between $W_{i}$ and $W_{j}$ is at most $c$.
\end{itemize}
\end{lemma}

\begin{proof}
We will prove the statement by induction on $s$ beginning with the base case $s=1$. Let $v_{1}, v_{2}, \dots, v_{n}$ be a $d$-degenerate ordering of the vertices of $H$. Partition $[N]$ into $n$ intervals, each of length at least $\frac{N}{2n}$. If the required ordering of the vertices of $H$ is $v_{i_1}, v_{i_2}, \dots, v_{i_n}$, then we 
will label the intervals in order as $V_{i_1}, V_{i_2}, \dots, V_{i_n}$. 

Consider the following process for embedding an ordered copy of $H$. We will embed vertices one at a time following the degenerate ordering, at step $t$ embedding the vertex $v_t$ into the set $V_t$. To this end, we will try to find, by induction, a sequence of vertices $w_1, w_2, \dots, w_t$ with $w_i \in V_i$ for $i = 1, 2, \dots, t$ and sets
$V_{i, t} \subset V_i$ for $i = t + 1, \dots, n$ satisfying the following properties:

\begin{itemize}

\item[1.] 
If $i, j \leq t$ and $v_i$ is adjacent to $v_j$, then $w_i$ is adjacent to $w_j$.

\item[2.]
If $j \leq t < i$ and $v_i$ is adjacent to $v_j$, then $w_j$ is adjacent to every vertex in $V_{i,t}$.

\item[3.] 
$|V_{i,t}| \ge c^{d_{i,t}} |V_i|$ for all $i > t$, where $d_{i,t}$ is the number of neighbors of $v_i$ in $\{v_1,\dots,v_t\}$.
\end{itemize}
If the process reaches step $n$, Property $1$ implies that mapping $v_i$ to $w_i$ for each $i = 1, 2, \dots, n$ gives the required ordered copy of $H$.

To begin the induction, we let $V_{i, 0} = V_i$ for all $i$. The properties stated above are then trivially satisfied. Suppose now that we have found $w_1, w_2, \dots, w_{t-1}$ and $V_{i,t-1}$ for all $i \geq t$ and we wish to define $w_{t}$ and $V_{i, t}$ for all $i > t$. Let $I_{t}=\{i > t :v_{i}\,\textrm{\textrm{is adjacent to}\,\ensuremath{v_{t}}}\}$
and note that $|I_{t}|\le\Delta$. We have
\[d_{i,t}=\left\{ \begin{split}d_{i,t-1}+1 & \quad \textrm{if }\, i\in I_{t},\\
d_{i,t-1} & \quad\textrm{if }\, i\notin I_{t}.
\end{split}
\right.\]

For every $i > t$ with $i\notin I_{t}$, let $V_{i,t}=V_{i,t-1}$
and note that these sets satisfy Properties $2$ and $3$ above. For
a vertex $w\in V_{t,t-1}$ and an index $i\in I_{t}$, let $V_{i,t}(w)$
be the set of neighbors of $w$ in $V_{i,t-1}$. If there exists
a vertex $w\in V_{t,t-1}$ such that 
\[
|V_{i,t}(w)|\ge c |V_{i,t-1}|\ge c^{d_{i,t-1}+1} |V_{i}| =c^{d_{i,t}} |V_i|
\]
 for all $i\in I_{t}$, then we may set $V_{i,t}=V_{i,t}(w)$ for
$i\in I_{t}$, take $w_t = w$ and proceed to the next step.
If this is not the case, then for each vertex $w$ in $V_{t,t-1}$ there
exists an index $i\in I_{t}$ for which $|V_{i,t}(w)|< c |V_{i,t-1}|$.
By the pigeonhole principle, there exists an index $i\in I_{t}$ such
that there are at least 
\[
\frac{1}{|I_{t}|}\cdot|V_{t,t-1}|\ge\frac{1}{\Delta} c^{d_{t,t-1}} |V_t|\ge c^{d_{t,t-1}}\frac{N}{2 \Delta n}
\]
vertices in $V_{t,t-1}$ which all have at most $c |V_{i,t-1}|$
neighbors in $V_{i,t-1}$. Let $W_{1}$ be these vertices and $W_{2} = V_{i,t-1}$. Then
\[
|W_{1}|, |W_{2}|\ge\frac{c^d N}{2\Delta n}.
\]
Relabeling $W_1$ and $W_2$ if necessary, we have found sets $W_{1}$ and $W_{2}$ satisfying conditions (i), (ii) and (iii) of the lemma. 
That is, if we cannot find an ordered copy of $H$, we can find sets $W_{1}$ and $W_{2}$ satisfying properties
(i), (ii), and (iii). In fact, $W_1$ and $W_2$ satisfy the following stronger conditions:
\begin{itemize}

\item[(i')] 
$|W_{1}|, |W_{2}|\ge\frac{c^d N}{2\Delta n}$,

\item[(ii')] 
all vertices in $W_{1}$ precede all vertices in $W_{2}$ or all vertices in $W_2$ precede all vertices in $W_1$,

\item[(iii')] 
each vertex in $W_{1}$ has at most $c |W_{2}|$ neighbors in $W_{2}$.

\end{itemize}

Suppose now that we are given an integer $s\ge2$ and the claim has been
proved for all smaller values of $s$. By the stronger form of the
case $s=1$ with $c'=2^{-s}c$, we can find two sets $W_{1}$ and $W_{2}$
such that $|W_{1}|,|W_{2}|\ge\frac{c^d N}{2^{sd+1} \Delta n}$ and all the
vertices in $W_{1}$ have at most $\frac{c}{2^{s}}|W_{2}|$ neighbors
in $W_{2}$. By applying the inductive hypothesis inside $W_{1}$, we
can find sets $W_{1,1},\cdots,W_{1,2^{s-1}}$ such that (i) for all $j$, $|W_{1,j}|\ge\frac{c^{(s-1)d}|W_1|}{(2^{(s-1)d + 1} \Delta n)^{s-1}} \geq \frac{c^{sd}N}{(2^{sd + 1} \Delta n)^s}$,
(ii) for $j<j'$, the vertices in $W_{1,j}$ precede the
vertices in $W_{1,j'}$, and (iii) for all $j$ and $j'$, the density
of edges between $W_{1,j}$ and $W_{1,j'}$ is at most $c$. 

For each $j=1,2,\ldots,2^{s-1}$, let $X_{2,j}\subset W_{2}$ be the set of vertices which have at least
$c |W_{1,j}|$ neighbors in $W_{1,j}$. By the degree
condition on vertices in $W_{1}$, we see that
\[
|X_{2,j}|\cdot c |W_{1,j}|\le|W_{1,j}|\cdot\frac{c}{2^s}|W_{2}|,
\]
and hence $|X_{2,j}|\le\frac{1}{2^s}|W_{2}|$. Therefore, $\sum_{j=1}^{2^{s-1}}|X_{2,j}|\le\frac{|W_{2}|}{2}$.

Let $W_{2}'=W_{2}\setminus\bigcup_{j=1}^{2^{s-1}}X_{2,j}$ and note
that $|W_{2}'|\ge\frac{|W_{2}|}{2}$. All vertices in $W_{2}'$ have
at most $c |W_{1,j}|$ neighbors in $W_{1,j}$ for all
$j=1,\cdots,2^{s-1}$. Now apply the $s-1$ case to $W_{2}'$ to find
sets $W_{2,1},\cdots,W_{2,2^{s-1}}$. We have, for all $j$, 
\[
|W_{2,j}|\ge\frac{c^{(s-1)d} |W_{2}'|}{(2^{(s-1)d + 1} \Delta n)^{s-1}}\ge\frac{c^{sd} N}{(2^{sd + 1} \Delta n)^{s}},
\]
and (relabeling $W_1$ and $W_2$, if necessary) one can easily check that the sets $W_{1,1}, \cdots, W_{1,2^{s-1}},$ $W_{2,1}, \cdots, W_{2,2^{s-1}}$
satisfy the claimed properties.
\end{proof}

Before moving on to the next ingredient, we note a corollary which we will need later on. 

\begin{corollary} \label{cor:dense}
Let $H$ be an ordered $d$-degenerate graph on $n$ vertices. Suppose that a real number $0 < c < \frac{1}{2}$ is given and that $N \geq (n^2 c^{-7d})^{4 \log(1/c)}$. If an ordered graph on $[N]$ does not contain an ordered copy of $H$, then there is a subset of order at least $(n^2 c^{-7d})^{-4 \log(1/c)}N$ with edge density at most $c$.
\end{corollary}

\begin{proof}
Let $s = \lceil \log(2/c) \rceil \leq 4 \log(1/c)$.
Since 
\[
	N \geq (n^2 c^{-7d})^{4 \log (1/c)} \geq (2\Delta n(2^s (2c^{-1}))^d)^s,
\]
we may apply Lemma~\ref{lem:deg} with $c$ replaced by $c/2$. This gives $t \geq 2/c$ sets $W_1, \dots, W_t$ satisfying $|W_1|, \dots, |W_t| \geq (n^2 c^{-7d})^{-4 \log(1/c)} N$ and such that the density of edges between $W_i$ and $W_j$ for all $i < j$ is at most $c/2$. For $1 \le i \le t$, let $W'_i$ be a subset of $W_i$ of cardinality exactly $N' := \lceil (n^2 c^{-7d})^{-4 \log(1/c)} N \rceil$ chosen
independently and uniformly at random. A first moment calculation now shows that there is a collection of subsets $W'_1 \subseteq W_1, \dots, W'_t \subseteq W_t$ such that $|W'_1| = \dots = |W'_t| = N'$ and $\sum_{i < j} e(W'_i, W'_j) \leq \frac{c}{2} \binom{t}{2} N'^2$. We claim that $\bigcup_{i=1}^t W'_i$ satisfies the required condition. To see this, note that the number of edges in this set is at most
\[\frac{c}{2} \binom{t}{2} N'^2 + t \binom{N'}{2} \leq \left(\frac{c}{2} + \frac{1}{t}\right) \binom{t N'}{2}.\]
Since $t \geq 2/c$, the claim follows.
\end{proof}

Lemma~\ref{lem:deg} tells us that if the edges of a graph on vertex set $[N]$ are two-colored in red and blue and there is no red copy of a particular $d$-degenerate ordered graph $H$ then it contains an ordered collection of large subsets with low red density between each pair of subsets. The next lemma shows that in this situation the blue graph, which has high density between these subsets, must contain a large trivially ordered multipartite graph.

\begin{lemma} \label{lem:chipart}
\label{lem:degenerate2}Let $K=K_{n,n,\cdots,n}$ be a trivially ordered
complete $\chi$-partite graph. If a graph on vertex set $[N]$ is such that 
there exist sets $W_{1},W_{2},\cdots,W_{\chi}$ satisfying the following conditions:
\begin{itemize}
\item[(i)] 
for all $i$, $|W_{i}|\ge4\chi n$,

\item[(ii)] 
for $i< j$, the vertices in $W_{i}$ precede the vertices in $W_{j}$,

\item[(iii)] 
for $i<j$, the density of non-edges between $W_{i}$ and $W_{j}$ is at most $\frac{1}{8\chi^2 n}$,
\end{itemize}
then the graph contains a copy of $K$.
\end{lemma}

\begin{proof}
For each pair of distinct $i,j\in[\chi]$, define $W_{i,j}\subset W_{i}$
as the set of vertices which have at least $\frac{1}{4 \chi n}|W_{j}|$
non-neighbors in $W_{j}$. By property (iii), we see that 
\[|W_{i,j}|\cdot\frac{1}{4 \chi n}|W_{j}|\le\frac{1}{8\chi^2 n}|W_{i}||W_{j}|,\]
and thus $|W_{i,j}|\le\frac{1}{2\chi}|W_{i}|$. For
each $i$, define 
\[
W_{i}'=W_{i}\setminus\big(\bigcup_{j\neq i}W_{i,j}\big),
\]
and note that $|W_{i}'|\ge|W_{i}|-\sum_{j\neq i}|W_{i,j}|\ge\frac{1}{2}|W_{i}|\ge2 \chi n$.
For distinct $i,j$, each vertex in $W_{i}'$ has at most $\frac{1}{4\chi n}|W_{j}|\le\frac{1}{2\chi n}|W_{j}'|$
non-neighbors in $W_{j}'$.

Let $v_{1},v_{2},\cdots, v_{\chi n}$ be the vertices of $K$ ordered as in
the trivial ordering
and let $\sigma(v_{i})\in[\chi]$ be the index of the part containing $v_{i}$. We will embed the vertices of $K$ one at a time. At the $t$-th
step, we map $v_{t}$ to a vertex $w_{t} \in W'_{\sigma(v_{t})}$.
Note that there are at most $\chi n$ neighbors of $v_{t}$ in $\{v_{1},v_{2},\cdots,v_{t-1}\}$.
Each such neighbor can forbid at most $\frac{1}{2 \chi n}|W'_{\sigma(v_{t})}|$
vertices from being the image of $w_{t}$. Also, there are at most $n-1$ vertices in $W'_{\sigma(v_{t})}$ already used for embedded vertices. Therefore, the number of
possible images of $v_{t}$ in $W'_{\sigma(v_{t})}$ is at least
\[
\big|W'_{\sigma(v_{t})}\big|-(n-1)-\chi n\cdot\frac{1}{2\chi n}|W'_{\sigma(v_{t})}|\ge\frac{1}{2}|W'_{\sigma(v_{t})}|-(n-1)\ge1.
\]
Thus, we can find a vertex $w_{t}\in W'_{\sigma(v_{t})}$ for which
$\{w_{1},\cdots,w_{t}\}$ is a copy of $K[\{v_{1},\cdots,v_{t}\}]$.
In the end, we find a copy of $K$. 
\end{proof}

By combining Lemmas~\ref{lem:deg} and \ref{lem:chipart}, it is now straightforward to prove Theorem~\ref{thm:uppergen}.

\begin{proof}[Proof of Theorem \ref{thm:degenerate}.]
Suppose that the edges of $[N]$ have been two-colored in red and blue and the graph does not contain a red
copy of $H$. Apply Lemma \ref{lem:deg} with $c = \frac{1}{D}$, where $D = 8\chi^2 n'$,
and $s=\lceil\log \chi\rceil$ to obtain sets $W_{1},\cdots,W_{\chi}$
such that
\begin{itemize}
\item[(i)] 
for all $i$, $|W_{i}|\ge\frac{c^{sd} N}{(2^{sd+1} \Delta n)^{s}} = \frac{N}{\left(2^{s d+1} \Delta n D^{d}\right)^s} \ge D \geq 4 \chi n'$,

\item[(ii)] 
for $i<j$, the vertices in $W_{i}$ precede the vertices in $W_{j}$, and

\item[(iii)] 
for $i<j$, the density of red edges between $W_{i}$ and $W_{j}$ is at most $\frac{1}{D} \leq \frac{1}{8 \chi^2 n'}$.
\end{itemize}
We may therefore apply Lemma \ref{lem:degenerate2} to find a blue copy of $K$, completing the proof of the first part.

To prove that $\rord(H) \leq n^{32d \log \chi}$ for any ordered $d$-degenerate graph on $n$ vertices with interval chromatic number $\chi$, we may clearly assume that $n \geq 3$. Since $H$ is contained in the trivially ordered complete $\chi$-partite graph $K_{n,n, \dots, n}$, we may apply the bound from the first part with $n' = n$. Noting that $s \leq 2 \log \chi$ and $D \leq 8n^3 \leq n^5$, we have
\begin{align*}
\rord(H) & \leq 2^{s^2 d + s} \Delta^s n^s D^{ds + 1} \leq 2^{2 d s^2} n^{2s} D^{2ds}\\
& \leq \chi^{8 d \log \chi} n^{4 \log \chi} n^{20 d \log \chi} \leq n^{32 d \log \chi},
\end{align*}
as required.
\end{proof}

We conclude this section by proving Theorem~\ref{thm:introdense}. We will need the following result of Erd\H{o}s and Szemer\'edi \cite{ES72}, which says that if a graph has low density, then it must contain a larger clique or independent set than would be guaranteed by Ramsey's theorem alone.

\begin{lemma} \label{ErdosSzem}
There exists a positive constant $a$ such that any graph on $N$ vertices with density $c \leq 1/2$ contains a clique or an independent set of order at least $a \frac{\log N}{c \log (1/c)}$. 
\end{lemma}

Using this lemma and Corollary~\ref{cor:dense}, it is now easy to prove the following strengthening of Theorem~\ref{thm:introdense}.

\begin{theorem}
There exists a constant $C$ such that if $H$ is an ordered graph on $n$ vertices with degeneracy $d$, then 
$$r_{<}(H,K_n) \leq 2^{C d \log^2(2n/d)}.$$
\end{theorem}

\begin{proof}
Let $c = d/n$, $a$ be as in Lemma~\ref{ErdosSzem} and $N = \max\{(n^2 c^{-7d})^{8 \log(1/c)}, 2^{2 n c \log(1/c)/a}\}$. Note that if $c \geq 1/2$, the result follows since $r_<(H,K_n) \le r(K_n,K_n) \le 2^{2n} \leq 2^{C d \log^2(2n/d)}$ for a large enough constant $C$. We may therefore assume that $c < 1/2$. Suppose now that the edges of the complete graph on vertex set $[N]$ are colored with two colors, say, red and blue. If there is no red copy of $H$, Corollary~\ref{cor:dense} tells us that there is a subset of size at least $(n^2 c^{-7d})^{-4 \log(1/c)} N \geq N^{1/2}$ with density at most $c$ in red. But then, by Lemma~\ref{ErdosSzem}, the graph contains either a red or a blue clique of order at least $a\frac{\log N^{1/2}}{c \log (1/c)} \geq n$. The result follows by noting that $N \leq 2^{C d \log^2(2n/d)}$ for a constant $C$ depending only on $a$. 
\end{proof}

\section{Off the diagonal} \label{sec:offdiag}

In this section, we will prove Theorem~\ref{thm:offdiag}, that there exists an ordered matching $M$ such that $r_<(M, K_3) \ge n^{4/3 - o(1)}$. We will show that this holds when $M$ is a {\it jumbled} matching, which we define to be an ordered matching on $[n]$ satisfying the following two properties:
\begin{itemize}
\item
every pair of disjoint intervals, each of order at least $2 \sqrt{n}$, have at least one edge between them; 

\item
every pair of disjoint intervals, each of order at most $2 \sqrt{n}$, have at most $9$ edges between them.
\end{itemize}

To construct such a matching, let $n = t^2$ be an even positive integer. Partition $[n]$ into $t$ intervals $I_1, \dots, I_t$, each of order $t$. There is an ordered matching $M$ on $[n]$ such that each pair $I_i$ and $I_j$ of distinct intervals has an edge between them. Indeed, this can be easily seen by greedily picking the edges between the intervals. If $A$ and $B$ are disjoint intervals, each of length at least $2t$, then there must be at least one edge between them, since each of $A$ and $B$ must completely contain one of the intervals $I_1, \dots, I_t$. Moreover, if $A$ and $B$ each have length at most $2t$ then there are at most $9$ edges between them, since each of $A$ and $B$ intersect at most $3$ of the intervals $I_1, \dots, I_t$. Therefore, by considering the largest integer $t$ such that $n \geq t^2$, Theorem~\ref{thm:offdiag} follows as an immediate corollary of Theorem~\ref{thm:offdiagjumb} below. Before tackling this theorem, we recall the famous Lov\'asz local lemma, which is a key tool in the proof (see, e.g., \cite{AlSp}).

\begin{lemma}
Let $A_1, \dots, A_n$ be events in an arbitrary probability space.
A directed graph $D=(V,E)$ on the set of vertices $V=[n]$
is called a dependency digraph for the events $A_1, \ldots, A_n$
if, for each $i \in [n]$, the event $A_i$ is mutually independent of
all the events $\{A_j:(i,j)\notin E\}$. Suppose that $D=(V,E)$
is a dependency digraph for the above events and suppose that there are
real numbers $x_1,\ldots,x_n$ such that $0 \le x_i < 1$ and
$\mathbb{P}(A_i) \le x_i \prod_{(i,j)\in E} (1-x_j)$ for all $i \in [n]$.
Then $\mathbb{P}(\bigcap_{i=1}^{n} \overline{A_i}) \ge \prod_{i=1}^{n}(1-x_i)$.
In particular, with positive probability no event $A_i$ holds.
\end{lemma}

With this lemma in hand, we are now ready to prove the main result of this section.

\begin{theorem}  \label{thm:offdiagjumb}
There exists a positive constant $c$ such that if $M$ is a jumbled matching on $n$ vertices, then 
\[r_{<}(M, K_3) \ge c\Big(\frac{n}{\log n}\Big)^{4/3}.\]
\end{theorem}

\begin{proof}
For notational convenience, we will suppose that
$M$ is a jumbled matching with $800n\log n$ vertices and prove that
$n^{4/3}\le r_{<}(M, K_3)$ for $n$ sufficiently large.

Let $\mathcal{G}$ be a given family of ordered graphs on $n^{2/3}$ vertices, each
with at least $40n\log n$ edges, and suppose that $|\mathcal{G}|\le e^{n^{2/3}\log n}$.
We claim that there exists an edge coloring $c_{1}$ of the complete graph
over the vertex set $[n^{2/3}]$ in two colors, red and blue, that satisfies the following list of properties:

\begin{itemize}
\item there is no blue triangle;
\item there is no graph in $\mathcal{G}$ all of whose edges are colored red;
\item there is no red clique of order at least $20n^{1/3}\log n$.
\end{itemize}

We will prove this claim later. For now, we will assume that it holds and show how it implies the required result.
To this end, let $N=n^{4/3}$ and partition $[N]$ into consecutive intervals $V_{1},V_{2},\cdots,V_{n^{2/3}}$, each
of length $n^{2/3}$. Let $\Phi$ be the set of injective
embeddings of $M$ into $[N]$ that respect the given order of vertices
of $M$. For $\phi \in \Phi$, let $G(\phi)$ be the graph over the
vertex set $[n^{2/3}]$, where there exists an edge between $i$ and
$j$ if and only if there is an edge of $M$ between $\phi^{-1}(V_{i})$
and $\phi^{-1}(V_{j})$. Let 
\[
\mathcal{G}=\{G(\phi)\,:\,\phi\in\Phi,\, e(G(\phi))\ge40n\log n\}.
\]
Since the maps in $\Phi$ respect the order of vertices, any $G(\phi)$
can be described by the last vertex in $\phi^{-1}(V_{i})$ for each
$i\in[n^{2/3}]$. Therefore, for $n$ sufficiently large,
\[
|\mathcal{G}|\le{800n\log n+n^{2/3} \choose n^{2/3}}\le\left(e(800n^{1/3}\log n+1)\right)^{n^{2/3}}\le e^{n^{2/3}\log n}.
\]
By the claim above, we can find an edge coloring $c_1$ 
that satisfies the properties listed above with respect to the family $\mathcal{G}$. 
Let $c_{2}$ be the coloring of the complete graph on $[N]$ where we color the edges between
$V_{i}$ and $V_{j}$ with color $c_{1}(i,j)$ for $i,j\in[n^{2/3}]$.
We color all the edges within the sets $V_{i}$ red. Note that this
coloring contains no blue triangle, since $c_{1}$ does not contain
a blue triangle.

Suppose that for $\phi\in\Phi$, the graph $\phi(M)$ forms a red
copy of $M$. Let $W_{i}=V(\phi(M))\cap V_{i}$ for each $i$. Let 
$S\subset[n^{2/3}]$ be the set of indices $i$ for which $|W_{i}|\le2 \sqrt{n}$,
and let $L=[n^{2/3}]\setminus S$. Note that for a pair of indices $i,j\in L$,
since $|W_{i}|,|W_{j}|> 2 \sqrt{n}$ and $M$ is a jumbled
matching, there exists an edge of $M$ between $W_{i}$ and $W_{j}$.
Thus, $c_{1}(i,j)$ is red and we can conclude that the set $L$ forms
a red clique in $c_{1}$. Hence, since $c_1$ contains no clique of order $20 n^{1/3} \log n$, we see that $|L|\le20n^{1/3}\log n$ and
\[
\left|\bigcup_{i\in L}W_{i}\right|\le|L|\cdot n^{2/3}\le20n\log n.
\]

Deleting all edges of $M$ with an endpoint in $W_i$ for some $i \in L$, we see that there are at least $360n\log n$ red edges within $\bigcup_{i\in S}W_{i}$.
For two indices $i,j\in S$, there are at most $9$ edges of $M$
between $W_{i}$ and $W_{j}$. Hence, $G(\phi)$ has at least $40n\log n$
edges, implying that $G(\phi) \in \mathcal{G}$. However, $G(\phi)$ must be red in the coloring $c_{1}$
and this is a contradiction. Therefore, $\phi(M)$ cannot form a red
copy of $M$.

It remains to prove the claim. Recall that $\mathcal{G}$ is a given
family of graphs on $n^{2/3}$ vertices with at least $40n\log n$
edges and $|\mathcal{G}|\le e^{n^{2/3}\log n}$. Let $\mathcal{H}$ be a
family of graphs on $n^{2/3}$ vertices with exactly $40n\log n$ edges
so that for each $G \in \mathcal{G}$, there exists a graph $H \in \mathcal{H}$
satisfying $G \supseteq H$. We may choose $\mathcal{H}$ so that
$|\mathcal{H}| \le |\mathcal{G}|$. We seek a coloring satisfying:

\begin{itemize}
\item there is no blue triangle;
\item there is no graph in $\mathcal{H}$ all of whose edges are colored red;
\item there is no red clique of order at least $20n^{1/3}\log n$.
\end{itemize}

Consider a random coloring of the complete graph on $[n^{2/3}]$ obtained by coloring each
edge red with probability $1-\frac{1}{2n^{1/3}}$ and blue with probability
$\frac{1}{2n^{1/3}}$. Let $P_{i}$ be the events corresponding to
blue triangles, $Q_{i}$ the events corresponding to copies of graphs in
$\mathcal{H}$ being red, and $R_{i}$ the events corresponding to red cliques.
Let $I_{P},I_{Q},I_{R}$ be the index sets for the events $P_{i},Q_{i},R_{i}$,
respectively. By applying the local lemma, we will prove that there exists a
coloring where none of the events $P_{i},Q_{i},R_{i}$ occur. Note that 
\[
|I_{P}|={n^{2/3} \choose 3}\le n^{2},\quad|I_{Q}|\le|\mathcal{H}|\le|\mathcal{G}|\le e^{n^{2/3}\log n},\quad|I_{R}|={n^{2/3} \choose 20n^{1/3}\log n}\le e^{20n^{1/3}\log^{2}n}.
\]
Let $x=\frac{1}{4n}$, $y=e^{-2n^{2/3}\log n}$ and $z=e^{-40n^{1/3}\log^{2}n}$.
For later usage, we note that 
\[
	y |I_Q| = o(1) \quad \text{and} \quad z|I_R| = o(1).
\]
The parameter used in the local lemma will be $x$ for events $P_i$, 
$y$ for events $Q_i$, and $z$ for events $R_i$.
We now check the conditions of the local lemma. 

\noindent {\bf Event $P_i$} : 
Since $P_i$ is an event depending on three edges, there are at most 
$3n^{2/3}$ other events $P_j$ that are adjacent to $P_i$ in the dependency graph.
For the events $Q_j$ and $R_j$, we use the trivial bound $|I_Q|$ and $|I_R|$ for
the number of events that depend on $P_i$. Hence, for $P_{i}$, we have
\begin{align*}
 & x\prod_{j\in I_{P},j\sim i}(1-x)\prod_{j\in I_{Q},j\sim i}(1-y)\prod_{j\in I_{R},j\sim i}(1-z)\\
 & = (1-o(1)) xe^{-x(3n^{2/3})}e^{-y|I_{Q}|}e^{-z|I_{R}|}\\
 & =(1-o(1))\frac{1}{4n} \ge \frac{1}{8n} = \mathbb{P}(P_{i}).
\end{align*}

\noindent {\bf Event $Q_i$} : 
Since $Q_i$ is an event depending on $40n\log n$ edges, there are at most 
$40n^{5/3}\log n$ events $P_j$ that are adjacent to $Q_i$ in the dependency graph.
Hence, for $Q_{i}$, we have
\begin{align*}
 &  y\prod_{j\in I_{P},j\sim i}(1-x)\prod_{j\in I_{Q},j\sim i}(1-y)\prod_{j\in I_{R},j\sim i}(1-z)\\
 & = (1-o(1)) ye^{-x(40n^{5/3}\log n)}e^{-y|I_{Q}|}e^{-z|I_{R}|}\\
 & = (1-o(1))e^{-2n^{2/3}\log n}e^{-10n^{2/3}\log n}\\
 & \ge e^{-20n^{2/3}\log n} \geq \left(1-\frac{1}{2n^{1/3}}\right)^{40n\log n} = \mathbb{P}(Q_{i}).
\end{align*}

\noindent {\bf Event $R_i$} : 
Since $R_i$ is an event depending on ${20n^{1/3}\log n \choose 2} \le 200n^{2/3}\log^2 n$ edges, there are at most $200n^{4/3}\log^2 n$ events $P_j$ that are adjacent to $R_i$ in the dependency graph. Hence, for $R_{i}$, we have
\begin{align*}
 &  z\prod_{j\in I_{P},j\sim i}(1-x)\prod_{j\in I_{Q},j\sim i}(1-y)\prod_{j\in I_{R},j\sim i}(1-z)\\
 & = (1-o(1)) ze^{-x(200n^{4/3}\log^{2}n)}e^{-y|I_{Q}|}e^{-z|I_{R}|}\\
 & = (1-o(1)) e^{-40n^{1/3}\log^{2}n} e^{-50n^{1/3}\log^{2}n}\\
 & \ge e^{-95n^{1/3}\log^{2}n} \ge \left(1-\frac{1}{2n^{1/3}}\right)^{\binom{20n^{1/3}\log n}{2}} = \mathbb{P}(R_{i}).
\end{align*}
The result follows.
\end{proof}

While we suspect that $r_{<}(M, K_3)\le n^{2-\epsilon}$ for some $\epsilon > 0$, we have been unable to improve on the trivial bound $r_<(M, K_3) \leq r(K_n, K_3) = O(n^2/\log n)$.

\section{Odds and ends} \label{sec:conclusion1}

\subsection{Connection to hypergraphs} \label{sec:hyper}

Given a $3$-uniform hypergraph $\mathcal{H}$, we define the Ramsey number $r(\mathcal{H})$ to be the smallest natural number $N$ such that every two-coloring of the edges of the complete $3$-uniform hypergraph $K_N^{(3)}$ contains a monochromatic copy of $\mathcal{H}$. In this subsection, we will show that for any $3$-uniform hypergraph $\mathcal{H}$, there is a family of ordered graphs $S_\mathcal{H}$ such that the Ramsey number of $\mathcal{H}$ is bounded in terms of the ordered Ramsey number of the family $S_\mathcal{H}$, where the ordered Ramsey number $\rord(\mathcal{F})$ for a family of ordered graphs $\mathcal{F}$ is defined to be the smallest $N$ such that every two-coloring of the edges of $\{1, 2, \dots, N\}$ contains an ordered copy of some $F \in \mathcal{F}$.

For any ordered graph $H$ on $\{1, 2, \dots, n\}$, we define a $3$-uniform hypergraph $\mathcal{T}(H)$ on vertex set $\{1, 2, \dots, n+1\}$ by taking all triples whose first pair is an edge of $H$. Given a $3$-uniform hypergraph $\mathcal{H}$ on $n + 1$ vertices, we let $S_\mathcal{H}$ be the collection of ordered graphs $H$ on $\{1, 2, \dots, n\}$ such that $\mathcal{H}$ is a subhypergraph of $\mathcal{T}(H)$. For example, if $\mathcal{H} = K_{n+1}^{(3)}$ then $S_\mathcal{H} = \{K_n\}$ and if $\mathcal{H} = K_{4}^{(3)}\setminus e$ then $S_\mathcal{H}$ contains three different graphs on vertex set $\{1, 2, 3\}$, namely, the complete graph $K_3$ and the two graphs with two edges containing the edge $\{1, 2\}$. Note that for any graph $H$, we have $H \in S_{\mathcal{T}(H)}$. Our main theorem relating upper bounds for Ramsey numbers of $3$-uniform hypergraphs to ordered Ramsey numbers is now as follows. 

\begin{theorem} \label{thm:hyper}
Let $\mathcal{H}$ be a $3$-uniform hypergraph. Then
\[r(\mathcal{H}) \leq 2^{\binom{\rord(S_\mathcal{H})}{2}}+1.\] 
\end{theorem}

\begin{proof}
The result follows from the method of Erd\H{o}s and Rado. Suppose that $N = 2^{\binom{\rord(S_\mathcal{H})}{2}}+1$ and the edges of the complete $3$-uniform hypergraph on $\{1, 2, \dots, N\}$ have been two-colored, say by red and blue. We will find an increasing sequence of vertices $v_1, v_2, \dots, v_{t+1}$ with $t = \rord(S_\mathcal{H})$ such that, for any given $i$ and $j$ with $i < j$, all triples of the form $\{v_i, v_j, v_k\}$ with $k > j$ have the same color $\chi(i, j)$. Consider the two-coloring of the edges of the complete graph on $v_1, v_2, \dots, v_t$ where the edge $\{v_i,v_j\}$ receives color $\chi(i, j)$. Then, since $t = \rord(S_\mathcal{H})$, this graph must contain a monochromatic copy of some ordered graph in $S_\mathcal{H}$. By construction, the $3$-uniform hypergraph on $v_1, v_2, \dots, v_{t + 1}$ must then contain a monochromatic copy of $\mathcal{H}$ in the same color. It only remains to find the sequence $v_1, v_2, \dots, v_{t + 1}$.

We will prove, by induction, that for any $1 \leq \ell \leq t$, there is a sequence of vertices $v_1, v_2, \dots, v_{\ell}$ and a set $V_{\ell}$ with 
\[|V_\ell| \geq 2^{\binom{\rord(S_\mathcal{H})}{2} - \binom{\ell}{2}}\] 
such that, for any $1 \leq i < j \leq \ell$, all triples $\{v_i, v_j, w\}$ with 
$w \in \{v_{j+1}, \ldots, v_{\ell}\} \cup V_\ell$ have the same color $\chi(i, j)$ depending only on $i$ and $j$. 

To begin, we let $v_1 = 1$ and $V_1 = \{2, 3, \dots, N\}$. Suppose now that $v_1, v_2, \dots, v_{\ell}$ and $V_{\ell}$ have been constructed satisfying the required conditions and we wish to find $v_{\ell + 1}$ and $V_{\ell + 1}$. We let $v_{\ell + 1}$ be the smallest element of $V_{\ell}$. Let $V_{\ell, 0} = V_{\ell} \setminus \{v_{\ell + 1}\}$. We will construct a sequence of subsets $V_{\ell,0} \supset V_{\ell, 1} \supset \dots \supset V_{\ell, \ell}$ such that, for each $j$, all triples $\{v_j, v_{\ell + 1}, w\}$ with $w \in V_{\ell, j}$ are of the same color that depends only on the index $j$.
Note that since $V_{\ell, \ell} \subset \cdots \subset V_{\ell,0} \subset V_{\ell}$ it follows that all triples $\{v_i, v_j, w\}$ with $1 \leq i < j \leq \ell+1$ and $w \in V_{\ell, \ell}$ have the same color depending only on the values of $i$ and $j$.

Suppose now that $V_{\ell, j}$ has been constructed in an appropriate fashion. To construct $V_{\ell, j+1}$, we consider the set of vertices $w \in V_{\ell,j}$ for which $\{v_{j+1}, v_{\ell+1}, w\}$
have color red. If this set has size at least $|V_{\ell, j}|/2$, we let $V_{\ell, j+1}$ be this set. Otherwise, we let $V_{\ell, j+1}$ be the complement of this set in $V_{\ell, j}$. In either case, $|V_{\ell, j+1}| \geq |V_{\ell, j}|/2$. To finish the construction of $V_{\ell + 1}$, we let $V_{\ell + 1} = V_{\ell, \ell}$. Note that for each $j \le \ell$ and every $w \in V_{\ell + 1}$ the triple $\{v_j, v_{\ell + 1}, w\}$ has the same color. Furthermore,
\[|V_{\ell + 1}| \geq \left\lceil \frac{|V_{\ell}|-1}{2^{\ell}} \right\rceil \geq 2^{\binom{\rord(S_\mathcal{H})}{2} - \binom{\ell + 1}{2}}.\]
In particular, $|V_t| \geq 1$ and choosing $v_{t+1} \in V_t$ completes the proof.
\end{proof}

We expect that the bound coming from this theorem is close to sharp in many cases. For example, for $\mathcal{H} = K_{n+1}^{(3)}$, the complete $3$-uniform hypergraph, we have $S_{\mathcal{H}} = \{K_n\}$, so Theorem~\ref{thm:hyper} returns the usual double-exponential bound, which is believed to be tight.  However, there are cases where the bound given above is far from the truth. This can be seen by considering the balanced complete tripartite $3$-uniform hypergraph $\mathcal{H}=K^{(3)}_{n,n,n}$. 
It has Ramsey number $2^{\Theta(n^2)}$ but the best bound the theorem above can give is
of the form $2^{2^{O(n)}}$ since all ordered graphs in $S_{\mathcal{H}}$ contain
the complete bipartite graph $K_{n,n}$ as a subgraph, thus implying that
$r_<(S_\mathcal{H}) \ge r(K_{n,n}) = 2^{\Omega(n)}$.

Another interesting example to consider is the hypergraph $\mathcal{T}(H)$ defined earlier. Since $H \in S_{\mathcal{T}(H)}$, Theorem~\ref{thm:hyper} has the following corollary.

\begin{corollary}
For any ordered graph $H$,
\[r(\mathcal{T}(H)) \leq 2^{\binom{r_<(H)}{2}} + 1.\]
\end{corollary}

This implies, for example, that if $M$ is an ordered matching on $\{1, 2, \dots, n\}$, then the Ramsey number of the hypergraph $\mathcal{T}(M)$ is at most $2^{n^{4 \log n}}$. For random matchings, we expect that this unusual looking bound is not far from the truth. We hope to discuss this direction further in future work.

\subsection{Linear ordered Ramsey numbers}

In this subsection, we characterize those graphs for which the ordered Ramsey number is linear in every ordering.
 A {\it vertex cover} $V$ of a graph $G$ is a set of vertices such that every edge of $G$ has an endpoint in $V$. The {\it cover number} $\tau(G)$ of a graph $G$ is the minimum size of a vertex cover of $G$. It is easy to show that the cover number of a graph is within a factor two of the size of the largest matching in the graph. We will prove the following theorem. 
 
 \begin{theorem}\label{charact}
 A graph $G$ on $n$ vertices has ordered Ramsey number $O(n)$ in every ordering if and only if $\tau(G)=O(1)$. 
 \end{theorem}

We first show that if an ordered graph has cover number $O(1)$, then its ordered Ramsey number is linear in the number of vertices. 

\begin{theorem}\label{coverfirst}
For every positive integer $\tau$, there exists a constant $c(\tau)$
such that every ordered graph $G$ on $n$ vertices
with a vertex cover of size $\tau$ satisfies
\[r_{<}(G) \leq c(\tau) n.\]
\end{theorem}

\begin{proof}
By Ramsey's theorem and a standard averaging argument, there is an integer $A$ and $a>0$ such that, for $N\ge A$, every two-coloring of the edges of  the complete graph on $[N]$ contains at least $a N^{2\tau+1}$ monochromatic copies of $K_{2\tau+1}$. Suppose now that we are given an 
ordered graph $G$ on $n$ vertices with a vertex
cover of size $\tau$. Let $N=\max\{A, \frac{2}{a} n\}$ and note that every
two-coloring of the edges of $[N]$ contains at least $a N^{2\tau+1}$ monochromatic
copies of $K_{2\tau+1}$. There exists a collection of vertices $v_{2},v_{4},\dots,v_{2\tau}$
such that at least $a N^{\tau+1}$ copies of $K_{2\tau + 1}$ have their $2$nd, $4$th$,\dots, (2\tau)$-th
vertex as $v_{2}, v_{4},\dots, v_{2\tau}$, in order. Without loss of
generality, we may assume that at least half of these copies are red. We let
$\mathcal{K}$ be this collection of red copies of $K_{2\tau + 1}$, noting that $|\mathcal{K}| \geq \frac{a}{2} N^{\tau + 1}$. We also let $v_{0}=0$ and
$v_{2\tau+2}=N+1$. 

We now show that we can find in this coloring a red copy of every ordered graph $G$ on $n$ vertices which has a vertex cover of size $\tau$. It follows that $r_{<}(G) \leq N$ so that we may take $c(\tau)=\max(A,\frac{2}{a})$. 
For $i=0,\dots,\tau$, let $X_{i}$ be the set of vertices $v$ between
$v_{2i}$ and $v_{2i+2}$ such that there exists a $K\in\mathcal{K}$
which has $v$ as its $(2i+1)$-th vertex. We see that
\[|\mathcal{K}|\le|X_{0}|\cdot|X_{1}|\cdot\cdots\cdot|X_{\tau}|.\]
Since $|X_{i}|\le N$ for all $i$ and $|\mathcal{K}|\ge \frac{a}{2} N^{\tau+1}$,
we see that $|X_{i}|\ge \frac{a}{2} N\ge n$ for all $i$. 
Note that the vertices in $X_{i}$ are incident to all vertices in
$v_{2},v_{4},\dots,v_{2\tau}$ by a red edge. Since $|X_{i}|\ge n$
for all $i$, we can find a red copy of $G$. Indeed, we can use the vertices $v_2,v_4,\ldots,v_{2\tau}$ as the $\tau$ vertices of the vertex cover and use common neighbors to embed the remaining vertices in the red copy of $G$.  
\end{proof}

Ordered Ramsey numbers are monotone. That is, if $H$ is an ordered subgraph of $G$, then $r_<(G) \geq r_<(H)$. The next theorem improves upon this simple fact if $H$ is a small subgraph of $G$. We will abuse notation slightly in the statement and proof by using $G$ and $H$ to refer to both the unordered graphs and particular ordered versions of these graphs.

\begin{theorem}\label{GHtheorem}
Let $G$ be a graph on $n$ vertices and $H$ a subgraph on $t$ vertices. Then, for every ordering of $H$, there is an ordering of $G$ such that $r_{<}(G) \geq \lfloor \frac{n}{t} \rfloor(r_{<}(H)-1)$. 
\end{theorem}
\begin{proof}
Let $s=\lfloor \frac{n}{t} \rfloor$. Given an ordering of $H$ with vertex set $[t]$, we consider any ordering of $G$ where vertex $i$ of $H$ is in place $1+(i-1)s$. This is chosen so that $H$ keeps its ordering within $G$ and any interval of $s$ vertices has at most one vertex from $H$.

Let $r = r_<(H) - 1$. By the definition of the ordered Ramsey number, there is an edge coloring $c$ of the complete ordered graph on vertex set $[r]$ containing no monochromatic copy of $H$ with the given ordering. We will construct a coloring of the complete ordered graph on $N=sr$ vertices with no monochromatic copy of $G$ with the ordering above. Partition the complete ordered graph on vertex set $[N]$ into $r$ intervals each with $s$ vertices. We color every edge between the $i$th interval and the $j$th interval with color $c(i,j)$ and color the edges inside each interval arbitrarily. 

We claim that this edge coloring of $[N]$ has no monochromatic copy of $G$ with the ordering defined above. Indeed, if there were such a copy of $G$, the vertices of the ordered subgraph $H$ would have to be in distinct intervals. This is because the ordering of $G$ was chosen so that no two vertices of $H$ are contained in an interval of $s$ vertices. However,  if the vertices of the monochromatic ordered copy of $H$ are in distinct intervals, then we would also get an ordered monochromatic  copy of $H$ in the edge coloring $c$ of $[r]$. This contradicts the definition of $c$ and completes the proof.  
\end{proof}

We will now use this result to show that if a graph on $n$ vertices has large cover number, then there is an ordering for which the ordered Ramsey number is large. 

\begin{corollary}\label{taucor}
There is a positive constant $c$ such that every graph $G$ with $n$ vertices and cover number $\tau$  has an ordering with ordered Ramsey number at least $\tau^{c\log \tau / \log \log \tau}n$. 
\end{corollary}
\begin{proof}
Since $G$ has cover number $\tau$, it contains a matching with $t \geq \tau$ vertices. Consider the induced subgraph $H$ of $G$ on these $t$ vertices. By Theorem \ref{thm:intromatchlower}, there is an ordering of $H$ with Ramsey number at least $t^{c\log t/\log \log t}$. By Theorem \ref{GHtheorem}, there is an ordering of $G$ with Ramsey number at least $\lfloor \frac{n}{t} \rfloor (t^{c\log t/\log \log t} - 1)$. By making $c$ a little smaller, we obtain the desired result.   
\end{proof}

If $G$ has $n$ vertices and cover number $\tau(G)=\omega(1)$, Corollary \ref{taucor} implies that there is an ordering of $G$ with ordered Ramsey number $\omega(n)$. Hence, Theorem \ref{charact} follows immediately from Theorem \ref{coverfirst} and Corollary \ref{taucor}. 

In fact, as a random ordering of the vertices of a graph will likely space out most of the vertices of a given matching and a random matching almost surely has superlinear ordered Ramsey number, a minor modification of the proof above gives the following stronger statement. It shows that if the cover number of a graph is $\omega(1)$, then almost all orderings will have superlinear ordered Ramsey number. We leave the details to the interested reader. 

\begin{proposition}\label{taucor2}
There is a positive constant $c$ such that if $G$ is a graph with $n$ vertices and cover number $\tau$,  then almost every ordering of $G$ has ordered Ramsey number at least $\tau^{c\log \tau / \log \log \tau}n$. 
\end{proposition}

\subsection{Density conditions for ordered matchings}

For many of the bounds in this paper, we prove stronger off-diagonal results. For example, when bounding the ordered Ramsey number of an ordered $d$-degenerate graph $H$ on $n$ vertices, we actually estimate $r_<(H, K_n)$. When bounding this quantity, we first try to embed $H$ greedily. If this fails, we show that there must be a large subset where the density of edges is at least $1 - \frac{1}{n}$, so that we can easily embed copies of $K_n$. However, this seems incredibly wasteful for bounding $r_<(H)$, given that unordered copies of $H$ already appear at density roughly $1 - \frac{1}{d}$. Somewhat surprisingly, this intuition is wrong. Indeed, we will now show the existence of a positive constant $c$ for which there is an ordered matching $M$ on $n$ vertices with interval chromatic number $2$ and an ordered graph $G$ on $N=2^{n^c}$ vertices with edge density at least $1-n^{-c}$ which does not contain $M$ as an ordered subgraph. Note that for any fixed ordered graph $H$ with interval chromatic number $2$, the ordered extremal number ex$_<(N, H)$, defined in the introduction, is $O(N^{2-\epsilon})$. However, even for a typical ordered matching, this result shows that this subquadratic behavior is only exhibited for rather large values of $N$. 

Let $H$ be an ordered graph with interval chromatic number $2$ and $k+\ell$ vertices with every edge going between $[k]$ and $[k+1,k+\ell]$. Let $B$ be an ordered graph with interval chromatic number $2$ and $N+M$ vertices with every edge going between $[N]$ and $[N+1,N+M]$. We say that $H$ is an interval minor of $B$ if there are partitions $[N]=I_1 \cup \dots \cup I_k$ and $[N+1,N+M]=I_{k+1} \cup \dots \cup I_{k+\ell}$ into intervals with the vertices in $I_i$ coming before the vertices in $I_j$ for $i<j$ such that if $(i,j)$ is an edge of $H$, then there is at least one edge in $B$ between $I_i$ and $I_j$.

\begin{proposition} 
There is an ordered matching $M$ on $2n$ vertices with interval chromatic number $2$ and an ordered graph $G$ on $2^{\Omega(n^{1/12})}$ vertices with edge density $1-O(n^{-1/6})$ which does not contain $M$ as an ordered subgraph. 
\end{proposition} 

\begin{proof}
Let $M$ be an ordered matching on $2n$ vertices with interval chromatic number $2$, where each part is partitioned into $n^{1/2}$ intervals of size $n^{1/2}$, with an edge between each interval in the first part and each interval in the second part. Such a matching can be constructed greedily. 

Let $k=\frac{1}{2}n^{1/3} -2$. In the paper \cite{F13}, the second author constructed an ordered bipartite graph $B$ with parts of size $t=2^{\Omega(k^{1/4})}$ and density $1-O(k^{-1/2})$ across the two parts which does not contain $J_k$ as an interval minor, where $J_k$ is the ordered graph of interval chromatic number $2$ on $k+k$ vertices whose edge set consists of all $k^2$ pairs $(i,j)$ with $1 \leq i \leq k < j \leq 2k$. We now construct a graph $G$ on $N=2n^{1/6} t = 2^{\Omega(n^{1/12})}$ vertices.
Partition the vertex set into $2n^{1/6}$ consecutive intervals $I_1, \ldots, I_{2n^{1/6}}$ of $t$ vertices each and place a copy of $B$ between each pair of intervals, while each of the $2n^{1/6}$ intervals forms an independent set. The resulting graph $G$ has density at least $1-1/(2n^{1/6})-O(k^{-1/2})=1-O(n^{-1/6})$. 

If $G$ contains $M$ as an ordered subgraph, then, by the pigeonhole principle, there is an interval $I_i$ containing at least $n/(2n^{1/6})=n^{5/6}/2$ vertices from the first part of $M$. Recall that the first part of the matching $M$ was partitioned into $n^{1/2}$ intervals of size $n^{1/2}$. Hence, $I_i$ will contain at least $\frac{1}{2}n^{5/6}/n^{1/2} - 2 = k$ such intervals. Similarly, there is an interval $I_j$ containing at least $k$ intervals of the second part of $M$. Therefore, the subgraph induced on $I_i \cup I_j$ contains $J_k$ as an interval minor. However, this contradicts our choice of $B$, completing the proof. 
\end{proof}

\subsection{Induced ordered Ramsey numbers}

A graph $H$ is said to be an induced subgraph of a graph $G$ if $V(H) \subset V(G)$ and two vertices of $H$ are adjacent if and only if they are adjacent in $G$. The induced Ramsey number $r^*(H)$ is defined to be the minimum $N$ for which there is a graph $G$ such that every two-coloring of the edges of $G$ contains a monochromatic induced copy of $H$. 

That these numbers exist was proved by several groups of authors in the early seventies, though the original methods did not give good bounds for $r^*(H)$. The first major breakthrough on quantitative estimates was due to Kohayakawa, Pr\"omel and R\"odl \cite{KPR98}, who proved that there exists a constant $c$ such that if $H$ is a graph on $n$ vertices, then $r^*(H) \leq 2^{c n \log^2 n}$ (see also \cite{FS08}). This result was recently improved in \cite{CFS12} to $r^*(H) \leq 2^{c n \log n}$, edging closer to resolving the well-known conjecture of Erd\H{o}s that $r^*(H) \leq 2^{c n}$.

It is also possible to define an ordered version of the induced Ramsey number, which we denote by $r_{<}^*(H)$. This is the smallest $N$ for which there is an ordered graph $G$ such that every two-coloring of the edges of $G$ contains a monochromatic ordered induced copy of $H$. A straightforward modification of the proof of Theorem 1.2 in \cite{CFS12} allows one to prove a bound for $r_<^*(H)$ which brings it in line with the bounds mentioned above. We omit the details.

\begin{theorem}
There exists a constant $c$ such that for any ordered graph $H$ on $n$ vertices,
\[r^*_<(H) \leq 2^{c n \log n}.\]
\end{theorem}

Moreover, by modifying the proof of Theorem 1.4 from \cite{FS08}, we may prove the following induced variant of Theorem~\ref{thm:introBEwithChi}. We again omit the details.

\begin{theorem}
There exists a constant $c$ such that for any ordered graph $H$ on $n$ vertices with degeneracy $d$ and interval chromatic number $\chi$, 
\[r_<^*(H) \leq n^{c d \log \chi}.\]
\end{theorem}

\section{Open problems} \label{sec:conclusion2}

Many difficult and interesting problems arose in our study of ordered Ramsey numbers. In this section, we would like to draw attention to a few of them.

Our first problem relates to the off-diagonal ordered Ramsey number $r_<(M,K_3)$. We have already mentioned the trivial upper bound $r_<(M, K_3) \leq r(K_n, K_3) = O(n^2/\log n)$ for any matching $M$ on $n$ vertices, while Theorem~\ref{thm:offdiag} shows that there are matchings $M$ on $n$ vertices such that $r_<(M, K_3) = \Omega((n/\log n)^{4/3})$. We are not sure where the truth should lie, though we expect that the upper bound is far from optimal.

\begin{problem} 
Does there exist an $\epsilon > 0$ such that any ordered matching $M$ on $n$ vertices satisfies $r_<(M,K_3) = O(n^{2 - \epsilon})$?
\end{problem}

In Theorem~\ref{thm:intromatchlower}, we proved that there are matchings $M$ with $n$ vertices for which $r_<(M) \geq n^{c \log n/\log \log n}$, while Theorem~\ref{thm:intromatchupper} showed that this result is not far from the truth, in that $r_<(M) \leq n^{\lceil \log n \rceil}$ for every ordered matching $M$ on $n$ vertices. It would be very interesting to close the gap between these two bounds.

\begin{problem}
Close the gap between the upper and lower bounds for ordered Ramsey numbers of matchings.
\end{problem}

One of our most elementary observations (see Theorem~\ref{thm:matchupper}) was that if $M$ is a matching on $n$ vertices with interval chromatic number $2$, then $r_<(M) \leq n^2$. On the other hand, Theorem~\ref{thm:introlowerwithChi} shows that there are matchings $M$ with $n$ vertices and interval chromatic number $2$ for which $r_<(M) = \Omega(n^2/\log^2 n \log \log n)$. It would again be interesting to close the gap between these two bounds.

\begin{problem}
Close the gap between the upper and lower bounds for ordered Ramsey numbers of matchings with interval chromatic number $2$.
\end{problem}

The $d$-dimensional cube is the graph on vertex set $\{0, 1\}^d$ where two vertices are connected by an edge if and only if they differ in exactly one coordinate. This is a $d$-regular bipartite graph with $2^d$ vertices. It is a well-known open problem to determine whether the Ramsey number of the cube is linear in the number of vertices, that is, whether $r(Q_d) = O(2^d)$. The best known upper bound~\cite{CFS16, L15} is quadratic in the number of vertices.

For any ordering of the cube, Theorem~\ref{thm:uppergen} easily implies that $\rord(Q_d) \leq 2^{c d^3}$,
while, for a random ordering, Theorem~\ref{thm:matchingrandom} and the fact that $Q_d$ contains a matching of size $2^{d-1}$ easily imply that 
$\rord(Q_d) \geq 2^{c' d^2/\log d}$. We believe that the lower bound is closer to the truth.

\begin{problem}
Improve the upper bound for the ordered Ramsey number of the cube.
\end{problem}

For orderings of the cube with interval chromatic number $2$, Theorem~\ref{thm:uppergen} again implies that $\rord(Q_d) \leq 2^{c d^2}$,
while Theorem~\ref{thm:introlowerwithChi} implies that there exist orderings such that $\rord(Q_d) \geq c' \frac{2^{2d}}{d^2\log d}$.
Again, it seems more likely that the lower bound is closer to the truth.

\begin{problem}
Improve the upper bound for the ordered Ramsey number of cubes with interval chromatic number $2$.
\end{problem}

Using a result of F\"uredi and Hajnal \cite{FH92}, Balko, Cibulka, Kr\'al and Kyn\v cl \cite{BCKK14} noted that there are orderings of the path on $n$ vertices for which the ordered Ramsey number is linear in $n$. It would also be interesting to decide whether a similar phenomenon holds for other graphs, that is, whether there are orderings where the ordered Ramsey number is close to the usual Ramsey number. In particular, it would be interesting to decide whether there are orderings of the cube for which the ordered Ramsey number is as small as the usual Ramsey number. 

\begin{problem}
Is there an ordering of the cube such that $r_<(Q_d) \leq 2^{cd}$ for some absolute constant $c$?
\end{problem}

We have characterized those graphs such that for every ordering of a graph $H$, the ordered Ramsey number is linear in $|H|$. These are precisely the graphs $H$ for which the edges of $H$ can be covered by $O(1)$ vertices. A problem of a similar nature is to determine which graphs have some ordering for which the ordered Ramsey number is linear. In particular, since the Ramsey number for graphs of bounded maximum degree is linear in the number of vertices, it is natural to ask whether there are always orderings of $d$-regular graphs for which the ordered Ramsey number is linear. We think this unlikely for random regular graphs.

\begin{problem}
Do random $3$-regular graphs have superlinear ordered Ramsey numbers for all orderings?
\end{problem}
We note that if the answer to this question is positive, as expected, it would also show that there are graphs for which the ordered Ramsey number is always significantly larger than the Ramsey number, regardless of the ordering.

We only have a rather poor understanding of ordered Ramsey numbers for more than two colors. For example, we could only prove the bound $r_<(M; q) \leq n^{(2 \log n)^{q-1}}$ for any matching $M$ on $n$ vertices. We believe that something much stronger should hold.

\begin{problem}
Show that for any natural number $q \geq 3$ there exists a constant $c_q$ such that $r_<(M;q) \leq n^{c_q \log n}$ for any matching $M$ on $n$ vertices.
\end{problem}

Finally, we recall Theorem~\ref{thm:bandwidth}, which says that for any natural number $k$ and any ordered matching $M$ on $n$ vertices with bandwidth at most $k$, the ordered Ramsey number $r_<(M)$ satisfies $r_<(M) \leq n^{\lceil \log k \rceil + 2}$. It would be very interesting to know whether this theorem can be extended to graphs other than matchings. That is, we have the following problem.

\begin{problem} \label{prob:bdw}
Show that for any natural number $k$ there exists a constant $c_k$ such that $r_<(H) \leq n^{c_k}$ for any ordered graph $H$ on $n$ vertices with bandwidth at most $k$.
\end{problem}

\noindent
{\bf Note added.} Before the final version of this article went to print, a solution to Problem~\ref{prob:bdw} was found by Balko, Cibulka, Kr\'al and Kyn\v cl and added to their paper~\cite{BCKK14}. Their result follows from an iterated application of Theorem~\ref{coverfirst}. To see this, consider the ordered graph $H_{t,k}$ on $2t + k$ vertices consisting of three successive intervals $L, M$ and $R$, with $|L| = |R| = t$ and $|M| = k$, where $M$ is a copy of the complete graph $K_k$ and every vertex in $L$ and $R$ is connected to every vertex in $M$. Since $k$ is fixed, Theorem~\ref{coverfirst} implies that every two-coloring of the edges of the complete graph on $\{1, 2, \dots, n\}$ contains a monochromatic copy of $H_{t,k}$ with $t \geq c n$, where $c > 0$ depends only on $k$. Letting $L'$ and $R'$ be the intervals corresponding to $L$ and $R$ in this copy of $H_{t,k}$, we see that we may again apply Theorem~\ref{coverfirst} to find a monochromatic copy of $H_{t',k}$ in each of $L'$ and $R'$, where $t' \geq c t$. Iterating this procedure $\Omega_k(\log n)$ times, it is possible to find a monochromatic subgraph which contains all graphs with $n^{\epsilon_k}$ vertices and bandwidth at most $k$. We omit the details, referring the interested reader instead to~\cite{BCKK14}.

\vspace{3mm}
\noindent
{\bf Acknowledgements.} We would like to thank the anonymous referees for a number of helpful remarks.

\end{document}